\documentclass[12pt,letterpaper]{article}
\usepackage[utf8]{inputenc}
\usepackage{amsmath}
\usepackage{amsfonts}
\usepackage{amssymb}
\usepackage{enumitem}
\usepackage{mathtools}
\usepackage{xcolor}
\usepackage{amsthm}
\usepackage{csquotes}
\usepackage{tikz}
\usepackage{float}
\usepackage{hyperref}
\usepackage{caption}
\usepackage{xcolor}
\usepackage{authblk}
\usepackage{mathrsfs}
\usepackage{fullpage}

\renewcommand{\phi}{\varphi}

\DeclarePairedDelimiter{\floor}{\lfloor}{\rfloor}

\newtheorem{lemma}{Lemma}[section]
\newtheorem{theorem}[lemma]{Theorem}
\newtheorem{corollary}[lemma]{Corollary}
\newtheorem{question}[lemma]{Question}
\newtheorem{proposition}[lemma]{Proposition}

\theoremstyle{definition}
\newtheorem{example}[lemma]{Example}
\newtheorem{definition}[lemma]{Definition}

\renewcommand{\subset}{\subseteq}
\newcommand{\Plarge}{\mathcal{P}_{\text{large}}}
\newcommand{\Pmed}{\mathcal{P}_{\text{med}}}
\newcommand{\Psmall}{\mathcal{P}_{\text{small}}}

\newcommand{\Vsmall}{\mathbf{V}_{\text{small}}}
\newcommand{\Vmed}{\mathbf{V}_{\text{med}}}
\newcommand{\Vlarge}{\mathbf{V}_{\text{large}}}
\newcommand{\Psq}{\mathcal{P}_{\square}}
\newcommand{\Pno}{\mathcal{P}_{\not\square}}

\date{\small{Department of Mathematics, University of California, Los Angeles (UCLA) 
\\ 
e-mail: \href{mailto:rushil@math.ucla.edu}{rushil@ucla.edu}}}
\title{\textbf{Sharp Bounds for Sets with Distinct Subset Products}}
\author{Rushil Raghavan}

\begin{document}
	
	\maketitle
	
	\begin{abstract} Let $A\subset [N]$ be such that for any pair of distinct subsets $B,C\subset A$, the products $\prod_{b\in B}b$ and $\prod_{c\in C}c$ are distinct. We prove that $|A|\leq \pi(N)+\pi(N^{1/2})+o(\pi(N^{1/2}))$, where $\pi$ is the prime counting function, answering a question of Erd\H{o}s. 
	\end{abstract} 	
	
\section{Introduction}
For $N\in\mathbb{N}$, let $f(N)$ denote the size of the largest subset $A$ of $[N]$ such that for any pair of distinct subsets $B,C\subset A$, $\prod_{b\in B}b \neq \prod_{c\in C}c$. We say that a set $A$ satisfying these conditions has \textit{distinct subset products}. Erd\H{o}s \cite{E1} initiated the study of the quantity $f(N)$, proving\footnote{See Definition \ref{asymptotic} for the asymptotic notation used in this paper.}
\[f(N) \leq  \pi(N) + O(\pi(N^{1/2})),\]
where $\pi(x) = |\{n\leq x:n\text{ is prime}\}|$ is the prime counting function.

He also produced the following example, establishing that the above bound is sharp up to the implicit constant:
\begin{example}\label{firstexmp} Let $A = \{p\leq N:p\text{ is prime}\}\cup \{p^2: p\text{ is prime and }p\leq N^{1/2}\}.$ Then $A$ has distinct subset products, so $f(N) \geq |A|=\pi(N)+\pi(N^{1/2})$.
\end{example}

He then asked \cite{E1}, \cite{E2} whether this estimate is optimal up to lower order terms. This problem is also listed at \url{https://www.erdosproblems.com/795}.
\begin{question}[Erd\H{o}s \#795]\label{mainconj} Is $f(N) = \pi(N)+\pi(N^{1/2})+o(\pi(N^{1/2}))?$
\end{question}
We answer this question affirmatively:
\begin{theorem}\label{mainthm} $f(N)  = \pi(N)+\pi(N^{1/2}) +O(N^{5/12})$.
\end{theorem}

In view of the above estimate, it is natural to ask for more precise information about the lower-order term. Erd\H{o}s also considered this in \cite{E1}, and speculated that a refinement of Example \ref{firstexmp} may be optimal. To be more specific, for $k\in\mathbb{N}$, let $g(k)$ be the smallest element of $\mathbb{N}$ for which there is an $E_k\subset [g(k)]$ of size $k$ such that for any distinct $F,F'\subseteq E_k$, $\sum_{n\in F}n\neq \sum_{n\in F'}n$. In other words, let $g(k)$ be the smallest possible maximal element in a set of size $k$ with distinct subset sums. Then, by setting 
\[A = \bigcup_{k=1}^\infty \bigcup_{n\in E_k}\{p^n:p\in (N^{1/g(k+1)},N^{1/g(k)}]\text{ is prime}\},\]
we see that the subset products of $A$ are distinct, and thus 
\[f(N)\geq \sum_{k=1}^\infty \pi(N^{1/g(k)}). \]
Since $g(1)=1$, $g(2)=2$, $g(3)=4$, and $g(4)=7$ (with $E_4=\{3,5,6,7\}$), Erd\H{o}s established $f(N)\geq \pi(N)+\pi(N^{1/2})+\pi(N^{1/4})+\pi(N^{1/7})$, and speculated that the above infinite sum may be best possible. However, we produce an example that improves upon this.
\begin{theorem}\label{lowerbound} $f(N)\geq \pi(N)+\pi(N^{1/2})+\frac{1}{3}\pi(N^{1/3})-O(1)$.
\end{theorem}

It is also natural to consider the additive variant of Question \ref{mainconj}, or in other words, to determine the asymptotic behavior of the function $g(k)$. This is another problem of Erd\H{o}s, which can be seen at \url{https://www.erdosproblems.com/1}. Although our methods do not address this problem, the interested reader may consult \cite{DFX}, \cite{S} for the best-known lower bounds on $g(k)$ and some history.

Since squares feature so prominently in Example \ref{firstexmp}, one may ask about what kinds of estimates can be obtained in the absence of squares. Our method answers this question as well, with a somewhat smaller second-order term.
\begin{theorem}\label{squarefree} Let $h(N)$ be the maximal size of a subset of $[N]$ consisting of squarefree integers with distinct subset products. Then $h(N)\leq \pi(N)+\frac{1}{2}\pi(N^{1/2}) + O(N^{5/12})$.
\end{theorem}

This estimate is also sharp up to the error term.
\begin{theorem}\label{lowerbound2} $h(N) \geq \pi(N)+\frac{1}{2}\pi(N^{1/2}) + o(\pi(N^{1/2}))$.
\end{theorem}

\subsection{Notation and strategy of proof}
\begin{definition}[Subset Product Set] Given a finite set $S\subset \mathbb{N}$, we define the \textit{subset product set} $\Pi(S)=\{\prod_{t\in T}t:T\subset S\}$. 
\end{definition}

Throughout the proof, we will use the fact that an element of $[N]$ can be divisible by at most one prime in $(N^{1/2},N]$, and at most two primes (with multiplicity) in $(N^{1/3},N]$. We thus define 
\begin{definition}[Small, Medium, and Large Primes, Valuations] Let $\Psmall$ denote the set of primes in $[N^{1/3}],$ $\Pmed$ the set of primes in $(N^{1/3},N^{1/2}]$, and $\Plarge$ the set of primes in $(N^{1/2},N]$. We will also use the terms ``small primes", ``medium primes", and ``large primes" to refer to elements of $\Psmall$, $\Pmed$, and $\Plarge$, respectively.

For a fixed prime $p\in[N]$, and $n\in\mathbb{N}$, let $V_p(n)$ denote the valuation of $n$ at $p$, i.e., the largest nonnegative integer $r$ such that $p^r$ divides $n$. Then define functions 
\[\Vsmall:\mathbb{N}\to \mathbb{Z}^{\Psmall},\quad \Vmed:\mathbb{N}\to \mathbb{Z}^{\Pmed}, \quad \Vlarge:\mathbb{N}\to \mathbb{Z}^{\Plarge}\quad\text{by}\] 
\[\Vsmall(n) = (V_p(n))_{p\in \Psmall}, \quad \Vmed(n) = (V_p(n))_{p\in \Pmed}, \quad \Vlarge(n) = (V_p(n))_{p\in \Plarge}. \]
We will also define $\Vlarge\times\Vmed:\mathbb{N}\to \mathbb{Z}^{\Plarge\cup\Pmed}$ as $(\Vlarge\times\Vmed)(n) = (V_p(n))_{p\in \Plarge\cup\Pmed}$.
\end{definition}

We will also use some standard asymptotic notation:
\begin{definition}[Big-O and Little-o Notation]\label{asymptotic} Given functions $F$, $G$, and $H$ from $\mathbb{N}$ to $[0,\infty)$, we say
\begin{itemize}
    \item $F(n) = O(G(n))$ if there is a constant $C>0$ such that $F(n)\leq CG(n)$ for all $n\in\mathbb{N}$,
    \item $F(n) = o(G(n))$ if for all $c>0$, there is an $N_0\in\mathbb{N}$ such that for all $n\geq N_0$, $F(n)\leq cG(N)$,
    \item $F(n) = G(n)+O(H(n))$ if $|F(n) - G(n)| = O(H(n))$.   
\end{itemize}
\end{definition}

To prove Theorem \ref{mainthm}, we will use a graph-theoretic approach. Our graphs will be simple, i.e., containing no loops or multiple edges. We will need to pay attention to certain subgraphs in our analysis:
\begin{definition}[Paths, Cycles, Circuits] Let $G = (V,E)$ be a graph. 

\begin{itemize}
    \item A path of length $k$ is a collection of edges of the form 
    
    $\{\{v_1,v_2\}, \{v_2,v_3\},\dots,\{v_{k},v_{k+1}\}\}\subset E$, where $v_1,\dots,v_{k+1}$ are distinct vertices in $V$.  
    \item A cycle of length $k$ is a collection of edges of the form
    
    $\{\{v_1,v_2\}, \{v_2,v_3\},\dots,\{v_{k-1},v_k\}, \{v_k,v_1\}\}\subset E$, where $v_1,\dots,v_k$ are distinct vertices in $V$.  
    \item A circuit of length $k$ is a collection of edges of the form
    
    $\{\{v_1,v_2\}, \{v_2,v_3\},\dots,\{v_{k-1},v_k\}, \{v_k,v_1\}\}\subset E$, where $v_1,\dots,v_k$ are (not necessarily distinct) vertices in $V$.  
    \end{itemize}
\end{definition}

In \cite{E1}, to prove $f(N)\leq \pi(N)+O(\pi(N^{1/2}))$, Erd\H{o}s counted the possible number of prime factorizations in elements of $\pi(A)$, where $A\subset [N]$ has distinct subset products. His proof is a counting argument based on the following three observations:
\begin{itemize}
    \item An element of $[N]$ is divisible by at most one large prime, and at most two large or medium primes (with multiplicity).
    \item The range of $\Vsmall$ on $\Pi([N])$ is small (see Proposition \ref{smallcontrol}).
    \item There are at most $\pi(N)$ elements $a$ of $A$ such that there is a prime $p$ for which $p$ divides $a$, but no other element of $A$ is a multiple of $p$. 
\end{itemize}
Optimizing his argument, one can obtain the bound $f(N)\leq \pi(N)+22\pi(N^{1/2})$. We will utilize each of these observations in our approach, but we take into account some more refined information as well. For example, if $p,q$ are large primes and $r,s$ are medium primes, the argument from \cite{E1} does not use the fact that $A$ cannot contain the elements $pr$, $qr$, $ps$, and $qs$ (although it does show that $A$ cannot contain this configuration for many primes $p,q,r,s$).

The strategy of our proof is as follows. Given a set $A$ with distinct subset products, we can produce a graph whose vertices correspond to large and medium primes, and whose edges correspond to (most) elements of $A$. We call this graph the \textit{prime factorization graph} of $A$. The condition that $A$ has distinct subset products restricts the number of short cycles that can appear in this graph. We can exploit this fact to prove an upper bound on the number of edges in this graph, and thus an upper bound on $|A|$. In Section 2, we show how to construct this prime factorization graph. In Section 3, we use the condition that $A$ has distinct subset products to remove circuits and cycles from the prime factorization graph without removing too many edges. In Section 4, we estimate the number of edges in a graph without any short cycles, and finally establish Theorems \ref{mainthm} and \ref{squarefree}.

In Section 5, we produce some new examples of nearly maximal sets with distinct subset products and prove Theorems \ref{lowerbound} and \ref{lowerbound2}.

\subsection*{Acknowledgements}
The author would like to thank Terence Tao for many helpful conversations and for introducing the author to this problem. The author would also like to thank Csaba S\'andor for observing that this argument gives an error term of $O(N^{5/12})$, as opposed to the estimate $O(N^{5/12+o(1)})$ written in an earlier version of this paper. 

\section{Constructing the Prime Factorization Graph}
Given a set $A$ with distinct subset products, we will construct a graph that encodes the large and medium prime factors of (most) elements of $A$. For this to be effective, we first need to control how much information is lost by ignoring small primes.
\begin{proposition}\label{smallcontrol} Let $\Pi([N])$ denote the subset product set of $[N]$. Then
\[|\Vsmall(\Pi([N]))| \leq \exp(O(N^{1/3})).\]
\end{proposition}
\begin{proof} For a fixed $p\in \Psmall$ and $n\in [N]$, $V_p(n)\in [0,\log_2(N)]$. Since an element of $\Pi([N])$ is a product of at most $N$ elements of $[N]$, for a fixed $n\in \Pi([N])$, $V_p(n)\in [0,N\log_2(N)]$. Thus $|\Vsmall(\Pi([N]))| \leq $
\[\prod_{p\in \Psmall}(N\log_2(N)+1)= \exp(|\Psmall|\log(N\log_2(N)+1)) = \exp(O(N^{1/3})),\]
where we use the estimate $|\Psmall| = O(N^{1/3}/\log(N))$ from the prime number theorem.
\end{proof}
In view of this proposition, if $A$ has distinct subset products, there cannot be many elements of $\Pi(A)$ with the same valuations at medium and large primes. In fact, the same is true for $A$.
\begin{proposition} Given $\mathbf{r}\in \mathbb{Z}^{\Plarge\cup\Pmed}$, let $A_{\mathbf{r}} = \{a\in A:(\Vlarge\times\Vmed)(a) = \mathbf{r}\}$.
Let $R = \{\mathbf{r}\in \mathbb{Z}^{\Plarge\cup\Pmed}:|A_\mathbf{r}|\geq 2\}$. Then 
\[\sum_{\mathbf{r}\in R}|A_\mathbf{r}|= O(N^{1/3}).\]
\end{proposition} 
\begin{proof} For each $\mathbf{r}\in R$, let $B_\mathbf{r}\subset A_{\mathbf{r}}$ and $C_\mathbf{r}\subset A_{\mathbf{r}}$ each have size $\floor{|A_{\mathbf{r}}|/2}$. Then 
\[\prod_{\mathbf{r}\in R}\prod_{b\in B_{\mathbf{r}}}b \quad\text{and}\quad  \prod_{\mathbf{r}\in R}\prod_{c\in C_{\mathbf{r}}}c\]
each have the same valuations at all large and medium primes. There are thus at most 
$\exp(O(N^{1/3}))$ elements of $\Pi(A)$ of the form $\prod_{\mathbf{r}\in R}\prod_{b\in B_{\mathbf{r}}}b$. There are at least 
\[\prod_{\mathbf{r}\in R}\binom{|A_\mathbf{r}|}{\floor{|A_\mathbf{r}|/2}} \geq \prod_{\mathbf{r}\in R}(1.1)^{|A_\mathbf{r}|}\]
elements of this form, where we used the inequality $\binom{n}{\floor{n/2}}\geq (1.1)^n$ for all $n\geq 2$. Thus 
\[(1.1)^{\sum_{\mathbf{r}\in R}|A_\mathbf{r}|}\leq \exp(O(N^{1/3})),\]
and taking logarithms gives the desired inequality.
\end{proof}

\begin{corollary}\label{graphconstruct} Given $A\subset [N]$ with distinct subset products, one can remove $O(N^{1/3})$ elements of $A$ to get a set $A'\subset A$ such that $\Vlarge\times\Vmed$ is injective when restricted to $A'$. By removing at most one more element, we may assume $(\Vlarge\times\Vmed)(a)\neq \mathbf{0}$ for all $a\in A'.$    
\end{corollary}
 Because of this corollary, we will assume here and in the next section that $\Vlarge\times\Vmed$ is injective and nonzero on $A$. Observe also that for $a\in A$, $a$ is divisible by at most one large prime, and at most two large or medium primes (with multiplicity). Thus the tuple $(\Vlarge\times\Vmed)(a)$ either consists of:
 \begin{itemize}
     \item a 1 at one index, with 0 at all other indices,
     \item a 1 at each of two indices, at most one of them from $\Plarge$, and a 0 at all other indices, or
     \item a 2 at an index from $\Pmed$ and a 0 at all other indices.
 \end{itemize}
 
 Under these conditions, we may define the prime factorization graph of a set $A$.
 
 \begin{definition}[Prime Factorization Graph] The prime factorization graph $G(A)$ associated to $A$ has vertex set $\Plarge\cup \Pmed \cup \{1\}$. We connect $1$ to a vertex $p\in \Plarge\cup\Pmed$ by an edge if there is an $a\in A$ such that $V_p(a)=1$ and $V_q(a)=0$ for all other primes $q\in \Plarge\cup\Pmed$. We connect $p,q\in \Plarge\cup \Pmed$ by an edge if there is an element of $A$ divisible by both $p$ and $q$. 
 \end{definition}
By the aforementioned discussion, there is a bijection between the edges of $G(A)$ and the elements of $A$ which are not divisible by the square of any medium prime. We will still need to consider those elements, so we define the following.
\begin{definition}[$\Psq$, $\Pno$] We define $\Psq = \{p\in \Pmed:p^2\mid a\text{ for some }a\in A\}$ and $\Pno = \Pmed\setminus\Psq$.
\end{definition}
Then $|A|$ is the sum of the number of edges in $G(A)$ and $|\Psq|$.

\bigskip

\begin{example} Suppose $N=121$, so $\Psmall = \{2,3\}$, $\Pmed = \{5,7,11\}$, and $\Plarge$ is the remaining set of primes in $[N]$. On the set $A=\{50, 105, 77, 55, 65, 26,51\} = \{2*5^2,3*5*7,7*11,5*11,5*13,2*13,3*17\}$, $\Vlarge\times\Vmed$ is injective and nonzero. Its prime factorization graph is (omitting the isolated vertices corresponding to the primes in $(17,121]$):
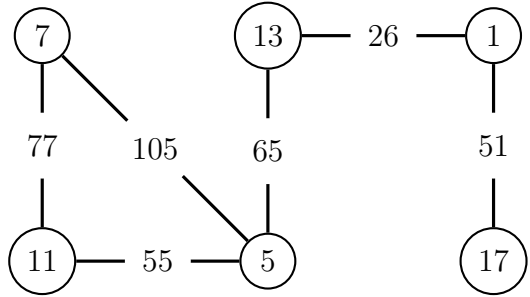
\begin{center}
\begin{tikzpicture}
    \begin{scope}[every node/.style={circle,thick,draw}]
    \node (A) at (3,0) {5};
    \node (B) at (0,3) {7};
    \node (C) at (0,0) {11};
    \node (D) at (3,3) {13};
    \node (E) at (6,3) {1};
    \node (F) at (6,0) {17};
\end{scope}

\begin{scope}[every node/.style={fill=white,circle},
              every edge/.style={draw = black, very thick}]
    \path  (A) edge node {$105$} (B);
    \path  (B) edge node {$77$} (C);
    \path (A) edge node {$55$} (C);
    \path (A) edge node {$65$} (D);
    \path (D) edge node {$26$} (E);
    \path (E) edge node {$51$} (F);
\end{scope}
\end{tikzpicture}
\end{center}
Each element of $A$ corresponds to an edge in the graph except for $50$, since $50$ is a multiple of $5^2$ and $5\in \Pmed$. 
\end{example}

\noindent\textbf{Remark.} Suppose $a\in A$ is divisible by $p^2$ for some $p\in \Pmed$. In our definition of $G(A)$, we do not include an edge corresponding to $a$. We could alternatively define the prime factorization graph so that $a$ would correspond to a loop from $p$ to itself. One can still make sense of the arguments in the following sections that way, but we found it clearer to express our graph-theoretic arguments using only simple graphs.

Having defined the necessary sets of primes, we will state our main estimate.
\begin{theorem}\label{mainestimate} Let $A\subset [N]$ have distinct subset products and let $\Psq = \\ \{p\in \Pmed: p^2\mid a\text{ for some }a\in A\}$. Then 
\[|A|\leq \pi(N)+\frac{1}{2}\pi(N^{1/2})+\frac{1}{2}|\Psq| + O(N^{5/12}).\]
\end{theorem}

Since $\Psq=\emptyset$ if the elements of $A$ are all squarefree and $|\Psq|\leq \pi(N^{1/2})$ for any set $A$, Theorems \ref{mainthm} and \ref{squarefree} follow immediately from Theorem \ref{mainestimate}.

\section{Cycle Removal}
Throughout this section, $A\subset [N]$ will have distinct subset products and be such that $\Vlarge\times\Vmed$ is injective and nonzero on $A$. Unless otherwise specified, the graph $G$ will denote the prime factorization graph of $A$. 

\begin{lemma} Let $C_1,\dots,C_n$ be the sets of edges in a maximal collection of edge-disjoint even circuits in $G$, each with length at most $2N^{1/12}$. Then $n = O(N^{1/3})$, so $\sum_{k=1}^n|C_k| = O(N^{5/12})$.
\end{lemma}
\begin{proof} For each $k\in[n]$, let $\{v_1^k,\dots,v_{|C_k|}^k\}$ be the set of vertices in the cycle $C_k$ indexed so that $\{\{v_1^k,v_2^k\},\{v_2^k,v_3^k\},\dots,\{v_{|C_k|}^k,v_1^k\}\}=C_k$. Let $A_0^k$ be the subset of $A$ corresponding to the edges $\{v_1^k,v_2^k\},\{v_3^k,v_4^k\},\dots,\{v_{|C_{k-1}|}^k, v_{|C_k|}^k\}$, and let $A_1^k$ be the subset of $A$ corresponding to all the other edges in $C_k$. Then 
\[(\Vlarge\times\Vmed)\left(\prod_{a\in A_{0}^k}a\right) = (\Vlarge\times\Vmed)\left(\prod_{a\in A_{1}^k}a\right).\]
Thus, for any choice of $\mathbf{\epsilon}\in \{0,1\}^n$,
\[\prod_{k=1}^n\prod_{a\in A_{\epsilon_k}^k}a\]
has the same valuation at all large and medium primes. By Proposition \ref{smallcontrol}, we thus have $2^n =\exp(O(N^{1/3}))$, as desired.
\end{proof}

\begin{example}
Suppose $A = \{15,55,84,154,221,247,323,551,437,667\}$, $\Psmall = \{2\}$, and $\Pmed\cup \Plarge = \{3,5,7,11,13,17,19,23,29$\}. Below is the prime factorization graph of $A$, containing two even circuits. The product of the elements on the blue edges and the product of the elements on the red edges have the same valuations at all large and medium primes. By swapping the red-blue labeling in either circuit, we can produce four subsets whose products have the same valuations at all large and medium primes.

    
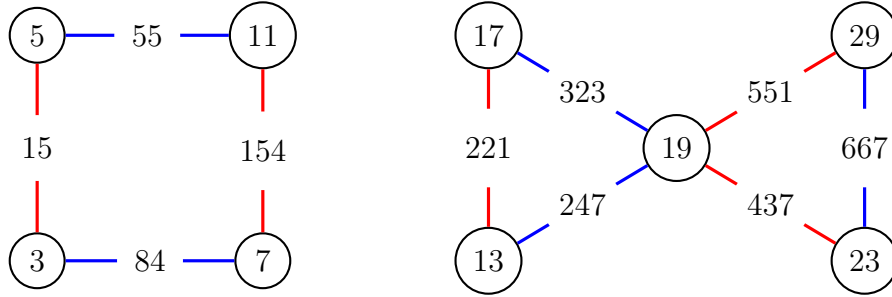
\centering
    
\begin{tikzpicture}
    \begin{scope}[every node/.style={circle,thick,draw}]
    \node (A) at (0,0) {3};
    \node (B) at (0,3) {5};
    \node (C) at (3,0) {7};
    \node (D) at (3,3) {11};

    \node (E) at (6,0) {13};
    \node (F) at (6,3) {17};
    \node (G) at (8.5, 1.5) {19};
    \node (H) at (11,0) {23};
    \node (I) at (11,3) {29};
\end{scope}
\begin{scope}[every node/.style={fill=white,circle},
              every edge/.style={draw=red,very thick}]
    \path  (A) edge node {$15$} (B);
    \path  (D) edge node {$154$} (C);
    \path (E) edge node {$221$} (F);
    \path (G) edge node {$551$} (I);
    \path (G) edge node {$437$} (H);
\end{scope}
\begin{scope}[every node/.style={fill=white,circle},
              every edge/.style={draw=blue,very thick}]
    \path (A) edge node {$84$} (C);
    \path (B) edge node {$55$} (D);
    \path (G) edge node {$247$} (E);
    \path (G) edge node {$323$} (F);
    \path (H) edge node {$667$} (I);
\end{scope}
\end{tikzpicture}

\end{example}

\begin{corollary}\label{removeeven} By removing $O(N^{5/12})$ edges from $G$, we may assume that $G$ has no even circuits of length at most $2N^{1/12}$.
\end{corollary}
\begin{proof} Let $C_1,\dots,C_n$ be the sets of edges in a maximal collection of edge-disjoint even circuits in $G$, each with length at most $2N^{1/12}$. We then have $\sum_{k=1}^n|C_k| = O(N^{5/12})$, so we may remove all the edges from $\bigcup_{k=1}^nC_k$ from $G$. After doing so, if there remains an even circuit of length at most $2N^{1/12}$ in $G$, this would contradict the maximality of $\{C_1,\dots,C_n\}$. 
\end{proof}

From here forward, we will assume that $G$ contains no even circuits of length at most $2N^{1/12}$. Having removed short even circuits from $G$, we now turn our focus to cycles with odd length. We first have the following proposition for a general graph $G$.

\begin{proposition}\label{disj} Let $G$ be a graph with no even circuits of length at most $2M$. Then any two odd cycles in $G$ of length at most $M$ are vertex-disjoint.
\end{proposition} 
\begin{proof} We will prove the contrapositive. Suppose $G$ contains two odd cycles with vertex sets  $V_1 = \{v_1,\dots,v_n\}$ and $V_2=\{w_1,\dots,w_m\}$, respectively, such that $v_1=w_1$. Supposing also that $n+m\leq 2M$, we will show that $G$ has an even circuit of length at most $2M$.  Let $E_1 = \{\{v_1,v_2\},\dots,\{v_n,v_1\}\}$ and $E_2 = \{\{w_1,w_2\},\dots,\{w_m,w_1\}\}$ be their respective edge sets. We consider two cases.

First, assume $E_1\cap E_2=\emptyset.$ Then $\{\{v_1,v_2\},\dots,\{v_n,v_1\},\{w_1,w_2\},\dots,\{w_m,w_1\}\}$ form a circuit of length $n+m$, which is even and at most $2M$.

On the other hand, suppose $E_1\cap E_2\neq \emptyset$. Since $E_1\neq E_2$, we may assume without loss of generality that $v_1=w_1$, $\{v_1,v_n\}=\{w_1,w_m\}$, but $\{v_1,v_2\}\neq \{w_1,w_2\}$. Let $j$ be the maximal index such that for all $1\leq i\leq j-1$, $\{v_i,v_{i+1}\}\notin E_2$. Let $P_1$ and $P_2$ be the edge sets from the two paths in $E_2$ from $v_1$ to $v_j$. Since $|P_1|+|P_2|$ is odd, we can combine one of $P_1$ or $P_2$ with the edges from $\{\{v_1,v_2\},\dots,\{v_{j-1},v_j\}\}$ to form an even circuit of length at most $n+m\leq 2M$. \end{proof}

\begin{lemma}\label{removeodd1} Let $C_1,\dots,C_n$ be the sets of edges in a maximal collection of edge-disjoint odd circuits in $G$, each with length at most $N^{1/12}$ and each having a vertex from $\Psq$. Then $n = O(N^{1/3})$, so $\sum_{k=1}^n|C_k| = O(N^{5/12}).$
\end{lemma}
\begin{proof} For each $k\in [n]$, let $\{v_1^k,\dots,v_{|C_k|}^k\}$ be the set of vertices in the cycle $C_k$ indexed so that $v_1^k\in \Psq$ and $\{\{v_1^k,v_2^k\},\{v_2^k,v_3^k\},\dots,\{v_{|C_k|}^k,v_1^k\}\}=C_k$. For each $k$, let $p_k$ be the prime corresponding to $v_1^k$, and let $a_k$ be the element of $A$ such that $p_k^2$ divides $a_k$. Note that by Proposition \ref{disj}, the elements $a_k$ are distinct for distinct $k$. Let $A_0^k$ be the set of elements of $A$ corresponding to the edges $\{v_1^k,v_2^k\},\{v_3^k,v_4^k\},\dots,\{v_{|C_k|-1}^k,v_{|C_k|^k}\}$, and let $A_1^k$ be the set of elements of $A$ corresponding to the other edges in $C_k$, union $\{a_k\}$. 

Then 
\[(\Vlarge\times\Vmed)\left(\prod_{a\in A_{0}^k}a\right) = (\Vlarge\times\Vmed)\left(\prod_{a\in A_{1}^k}a\right).\]
Thus, for any choice of $\mathbf{\epsilon}\in \{0,1\}^n$,
\[\prod_{k=1}^n\prod_{a\in A_{\epsilon_k}^k}a\]
has the same valuation at all large and medium primes. By Proposition \ref{smallcontrol}, we thus have $2^n =\exp(O(N^{1/3}))$, as desired.
\end{proof}

By a similar argument to Corollary \ref{removeeven}, by removing $O(N^{5/12})$ edges, we may assume without loss of generality that $G$ has no odd cycles of length at most $N^{1/12}$ with a vertex from $\Psq$.

\begin{example}

Suppose $A = \{15,65,84,154,143,9\}$, $\Psmall=\{2\}$, $\Pmed\cup\Plarge = \{3,5,7,11,13\}$, and $\Psq = \{3\}$. Below is the prime factorization graph of $A$. The product of the red elements (including $9$) and the product of the blue elements have the same valuations at all large and medium primes.
    
    
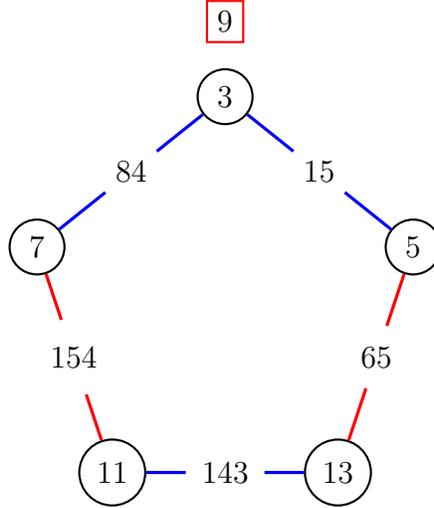
\centering

\begin{tikzpicture}
    \begin{scope}[every node/.style={circle,thick,draw}]
    \node (A) at (0,0) {11};
    \node (B) at (-1,3) {7};
    \node (C) at (3,0) {13};
    \node (D) at (4,3) {5};
    \node(E) at (1.5,5) {3};     
\end{scope}

\node[style={rectangle,thick,draw,red}] (sq) at (1.5,6) {\color{black}{9}};

\begin{scope}[every node/.style={fill=white,circle},
              every edge/.style={draw=red,very thick}]
    \path  (A) edge node {$154$} (B);
    \path  (D) edge node {$65$} (C);

\end{scope}
\begin{scope}[every node/.style={fill=white,circle},
              every edge/.style={draw=blue,very thick}]
    \path (A) edge node {$143$} (C);
    \path (B) edge node {$84$} (E);
    \path (E) edge node {$15$} (D);
\end{scope}
\end{tikzpicture}

\end{example}

\begin{lemma}\label{removeodd2} Let $E$ denote the edge set of $G$. There is a set of edges $E'\subset E$ of size at most $\frac{1}{2}(|\Pno|+1)$ such that the graph with edge set $E\setminus E'$ contains no cycles of length at most $N^{1/12}$. 
\end{lemma}
\begin{proof} Let $C_1,\dots,C_n$ be the sets of edges of all cycles of length at most $N^{1/12}$ in $G$. By Proposition \ref{disj}, we may assume that these sets are disjoint and that they share no vertices. None of these cycles have a vertex from $\Psq$, and no vertices in $\Plarge$ are adjacent, so each $C_k$ contains an edge between two vertices from $\Pno\cup\{1\}$. Since the vertex sets are disjoint, we thus have $2n\leq |\Pno\cup\{1\}|$. We may take $E'$ to consist of one edge from each cycle $C_k$. 
\end{proof}

\section{Proof of Theorems \ref{mainthm} and \ref{squarefree}}
Having removed all cycles of length at most $N^{1/12}$ from the prime factorization graph, we are ready to estimate the number of edges in it. We first record the following general estimate. This can be deduced from a result of Alon, Hoory, and Linial \cite{AHL}, but it is easier in our case, so we present a self-contained proof. The proof uses a standard breadth-first search argument; a similar argument can be found in \cite{Z}, Theorem 1.6.5, for example.

\begin{lemma}\label{edgecount} Let $G$ be a graph with $n$ vertices and at least $(1+c)n$ edges. Then $G$ has a cycle of length at most $\frac{2(c+1)}{c}(\log_2(n)+1)$.
\end{lemma}
\begin{proof} Without loss of generality, we may assume that $G$ is connected, since we may pass to a component satisfying $|V|\geq (1+c)|E|$. We can remove degree $1$ vertices and their incident edges while preserving the inequality $|V|\geq (1+c)|E|$, so we may assume without loss of generality that each vertex has degree at least 2. If $G$ contains a path of length greater than $\frac{c+1}{c}$, then we can remove all internal edges and vertices from that path while preserving the inequality $|V|\geq (1+c)|E|$, so we may assume that $G$ has no such paths. 

With all these assumptions set, we fix a vertex $v_0\in V$ and consider a breadth-first search starting at $v_0$. For each $k\in\mathbb{N}$, let $V_k = \{v\in V:\text{ the shortest path from $v$ to $v_0$ has length $k$}\}$. Fix $k\in\mathbb{N}$ with $1\leq k\leq \frac{c+1}{c}(\log_2(n)-1)$ and let $\ell$ be the least integer greater than $k+\frac{c+1}{c}$. 

If there is a $v\in v_\ell$ such that there are two distinct paths of length $\ell$ from $v$ to $v_0$, then we are done, since $v$ must thus contain a cycle of length at most $2\ell$. Otherwise, for each $v\in v_\ell$, let $\phi(v_\ell)\in V_k$ be the vertex from $V_k$ on the shortest path from $v_\ell$ to $v_0$. For each $v\in V_k$, we must have $|\phi^{-1}(v)|\geq 2$, since otherwise the path from $v$ to the only element in $\phi^{-1}(v)$ would be a path of length greater than $\frac{c+1}{c}$ consisting only of degree two vertices. We thus have $|V_\ell|\geq 2|V_k|$. Letting $m$ be the least integer greater than $\frac{c+1}{c}\log_2(n)$, we must thus have $|V_m|>2^{\log_2(n)}=|V|$, so there must be a $v\in V_m$ with at least two distinct paths from $v$ to $v_0$. There is thus a cycle of length at most $2m$ in $G$, as desired.
\end{proof}

We will apply Lemma \ref{edgecount} with $c = O(N^{-1/12+o(1)}).$ The following fact is required to ensure that the error term from Lemma \ref{edgecount} is multiplied by $\pi(N^{1/2})$ instead of $\pi(N)$ in our final estimate.
\begin{lemma}\label{boringcontrol} Let $G$ be the prime factorization graph of $A$, containing no even circuits of length at most $2N^{1/12}$. The set of vertices from $\Plarge$ with degree at least $2$ has size at most $\pi(N^{1/2}) + O(N^{5/12})$.
\end{lemma}
\begin{proof} Let $\mathcal{Q}\subset\Plarge$ be the set of vertices of degree at least $2$ in $G$. Recall that no pair of vertices in $\Plarge$ is connected by an edge, so if $p\in \mathcal{Q}$, then there are distinct $v_1,v_2\in \Vmed\cup \{1\}$ which are each connected to $p$ by an edge. 

Consider the graph with vertex set $\{1\}\cup\Pmed$, where $v_1$ is adjacent to $v_2$ if there is a path of length $2$ from $v_1$ to $v_2$ in $G$ with the middle vertex in $Q$. By construction, there is a surjection from the set of edges in this graph to $\mathcal{Q}$. Moreover, this graph is simple since $G$ contains no $4$-cycles.

Any cycle of length $m$ in this graph corresponds to a circuit of length $2m$ in $G$. This graph thus does not contain any cycles of length at most $N^{1/12}$. Applying Lemma \ref{edgecount} with $c = \frac{4\log_2(N)}{N^{1/12}}$, we thus find that the number of edges is at most $(1+O(N^{-1/12}\log(N)))(|\Pmed|+1)$, so $|\mathcal{Q}|\leq (1+O(N^{-1/12}\log(N)))\pi(N^{1/2})$, which is $\pi(N^{1/2}) + O(N^{5/12})$ by the prime number theorem. 
\end{proof}

We are now ready to prove Theorem \ref{mainestimate}, from which Theorems \ref{mainthm} and \ref{squarefree} follow immediately.
\begin{proof}[Proof of Theorem \ref{mainestimate}] Let $A\subset [N]$ have distinct subset products. By Corollary \ref{graphconstruct}, we may remove $O(N^{1/3})$ elements from $A$ to ensure that $\Vlarge\times\Vmed$ is injective and nonzero when restricted to $A$. We may now let $G = (V,E)$ be the prime factorization graph of $A$.

By Corollary \ref{removeeven} and Lemma \ref{removeodd1}, we may remove $O(N^{5/12})$ elements from $E$ to ensure that $A$ has no even circuits of length at most $2N^{1/12}$, and to ensure that $A$ has no odd cycles of length at most $N^{1/12}$ with a vertex from $\Psq$. Then, we may apply Lemma \ref{removeodd2} and remove $\frac{1}{2}|\Pno|+O(1)$ edges to remove all odd cycles of length at most $N^{1/12}$ from $G$. 

Let $\mathcal{Q}$ be the set of all vertices in $\Plarge$ with degree at least two. We remove all vertices from $\Plarge\setminus \mathcal{Q}$ and their incident edges. Let $V'$ and $E'$ be the remaining sets of vertices and edges after all these removals. We have 
\[|V'|\leq 1+\pi(N^{1/2})+|\mathcal{Q}|\]
and
\[|E'| \geq  |E| - (|\Plarge|-|\mathcal{Q}|) - \frac{1}{2}|\Pno| - O(N^{5/12}).\]
The graph $(V',E')$ has no cycles of length at most $N^{1/12}$. Applying Lemma \ref{edgecount} with $c = \frac{4\log_2(N)}{N^{1/12}}$, we find 
\[|E'| \leq (1+O(N^{-1/12}\log(N)))|V'|.\]
This results in the inequality
\[|E|-(|\Plarge|-|\mathcal{Q}|)  - \frac{1}{2}|\Pno| \leq (1+O(N^{-1/12}\log(N)))(\pi(N^{1/2})+|\mathcal{Q}|) + O(N^{5/12}) = \]\[\pi(N^{1/2})+|\mathcal{Q}|+O(N^{5/12}),\]
where we used Lemma \ref{boringcontrol} so that $O(N^{-1/12}\log(N))|\mathcal{Q}|=O(N^{5/12})$.

Finally, using the equation $|A|=|E|+|\Psq|+O(N^{1/3})$, we find \[|A|\leq 
\pi(N^{1/2}) + |\mathcal{Q}| + |\Plarge| - |\mathcal{Q}|+ \frac{1}{2}|\Pno|+|\Psq| + O(N^{5/12}) \]\[\leq  \pi(N)+\frac{1}{2}\pi(N^{1/2}) +\frac{1}{2}|\Psq| + O(N^{5/12}),\]
since $\Plarge = \pi(N)-\pi(N^{1/2})$ and $|\Psq|+|\Pno|\leq \pi(N^{1/2})$.
\end{proof}

\section{Examples}
In this section, we will discuss some examples of large sets with distinct subset products. 

Our graph-theoretic perspective enables us to produce some new examples of sets $A\subset [N]$ with distinct subset products and $|A| \approx  \pi(N)+\pi(N^{1/2})$.

\begin{example} Let $G$ be a tree with vertex set $\Pmed\cup \Psmall$. Let $A = \Plarge \cup \{p^2:p\in \Pmed\cup \Psmall\}\cup \{pq:p,q\in \Pmed\cup \Psmall\text{ are connected by an edge in $G$}\}$. Then $|A| = \pi(N)+\pi(N^{1/2})-1$ and $A$ has distinct subset products.
\end{example}

Next, we will produce an example of a set $|A|$ with distinct subset products and $|A|\geq \pi(N)+\pi(N^{1/2})+\frac{1}{3}\pi(N^{1/3})-O(1)$, establishing Theorem \ref{lowerbound}.
\begin{proof}[Proof of Theorem \ref{lowerbound}] Construct a set $A$ as follows. For any prime $p\in (N^{1/3},N]$, let $p\in A$. For any prime $p\in (N^{1/3},N^{1/2}]$, let $p^2\in A$. Let $p,q,r$ be distinct primes in $[N^{1/3}]$. The set $\{p^2q,p^2r,p^2,qr,p^3,q^3,r^3\}$ is a subset of $[N]$ with distinct subset products. We may thus divide the primes from $[N^{1/3}]$ into disjoint subsets of size $3$ (perhaps leaving one or two primes out) and for each triple $\{p,q,r\}$, add the $7$ elements $\{p^2q,p^2r,p^2,qr,p^3,q^3,r^3\}$ to $A$. The number of elements of $A$ is thus at least
\[\pi(N)-\pi(N^{1/2})+2(\pi(N^{1/2})-\pi(N^{1/3})) + 7\left(\frac{1}{3}\pi(N^{1/3})-O(1)\right) = \]\[\pi(N)+\pi(N^{1/2})+\frac{1}{3}\pi(N^{1/3})-O(1).\qedhere \]
\end{proof}

Finally, we may use our graph-theoretic perspective to produce a set of squarefree integers with distinct subset products and nearly $\pi(N)+\frac{1}{2}\pi(N^{1/2})$ elements, establishing Theorem \ref{lowerbound2}.
\begin{proof}[Proof of Theorem \ref{lowerbound2}]
 Let $\epsilon>0$, and let $N$ be sufficiently large (depending on $\epsilon$). By the prime number theorem, for every $k\in [1,1/\epsilon-2]$, the number of primes $p$ satisfying the inequality
\[N^{1/2}-(k+1)\epsilon N^{1/2}\leq p < N^{1/2}-k\epsilon N^{1/2} \]
is less than twice the number of primes $p$ satisfying
\[N^{1/2}+(k-1)\epsilon N^{1/2}< p \leq  N^{1/2}+k\epsilon N^{1/2}.\]
We may thus list all but at most $3\epsilon N^{1/2}$ of the primes in $[1,N^{1/2}]$ as $q_1,r_1,q_2,r_2,\dots,q_n,r_n$ so that there are distinct primes $p_1,\dots,p_n$ in $(N^{1/2},N]$ with $p_iq_i\leq N$, $p_ir_i\leq N$ for all $i$.

Let $\mathcal{Q}$ be the set of large primes not in $\{p_1,\dots,p_n\}$. Let \[A =\mathcal{Q}\cup \bigcup_{i=1}^n\{p_iq_i,p_ir_i,q_ir_i\} \cup \bigcup_{i=1}^n \{r_iq_{i+1}\}.\] Then $|A|\geq \pi(N)+\frac{1}{2}\pi(N^{1/2}) - 3\epsilon\pi(N^{1/2})-2$ and $A$ has distinct subset products.

The prime factorization graph of $A$ is as follows (omitting the vertices from $\mathcal{Q}$); it is constructed so that the argument from Lemma \ref{removeodd2} is sharp.
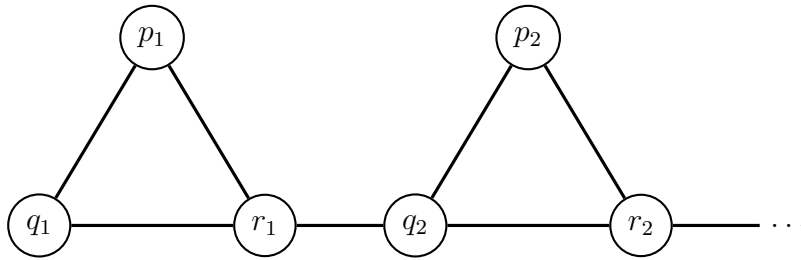
\begin{center}
    \begin{tikzpicture}
    \begin{scope}[every node/.style={circle,thick,draw}]
    \node (A) at (0,0) {$q_1$};
    \node (B) at (3,0) {$r_1$};
    \node (C) at (1.5,2.5) {$p_1$};
    \node (D) at (5,0) {$q_2$};
    \node (E) at (8,0) {$r_2$};
    \node (F) at (6.5,2.5) {$p_2$};
\end{scope}
\node (G) at (10,0) {$\dots$};

\begin{scope}[every node/.style={fill=white,circle},
              every edge/.style={draw=black,very thick}]
    \path  (A) edge (B);
    \path  (B) edge (C);
    \path (A) edge (C);
    \path (B) edge (D);
    \path (D) edge (E);
    \path (E) edge (F);
    \path (F) edge (D);
    \path (E) edge (G);
\end{scope}
\end{tikzpicture}
\end{center}

We have established that for all $\epsilon>0$, there is an $N_0\in\mathbb{N}$ such that for all $N\geq N_0$, $h(N)\geq \pi(N)+\frac{1}{2}\pi(N^{1/2})-3\epsilon\pi(N^{1/2})-2$. Thus, $h(N) \geq \pi(N)+\frac{1}{2}\pi(N^{1/2})-o(\pi(N^{1/2}))$. 
\end{proof}

\bibliographystyle{plain}
\bibliography{refs}

\end{document}